\documentclass{article}

\usepackage{amsmath,amssymb,amsthm}
\usepackage{graphicx}
\usepackage{hyperref}
\newtheorem{theorem}{Theorem}[section]
\newtheorem{lemma}[theorem]{Lemma}
\newtheorem{proposition}[theorem]{Proposition}

\newtheorem{definition}[theorem]{Definition}

\theoremstyle{remark}
\newtheorem{remark}[theorem]{Remark}

\title{The Equality Cases for the Grone--Merris-Bai Theorem}

\author{
Dongxiu Cai\thanks{Email: \texttt{diudiutse@sjtu.edu.cn}}
\and
Zhengbo Chen\thanks{Email: \texttt{czb911@sjtu.edu.cn}}
\and
Jia Yang\thanks{Email: \texttt{yangjia2020@sjtu.edu.cn}}
\and
Xiao-Dong Zhang\thanks{Corresponding author. Email: \texttt{xiaodong@sjtu.edu.cn}}
}

\date{
School of Mathematical Sciences, MOE-LSC, SHL-MAC,\\
Shanghai Jiao Tong University, Shanghai 200240, P.R. China
}

\begin{document}

\maketitle

\begin{abstract}
The Grone--Merris inequality, conjectured by Grone and Merris~(1994) and
first proved by Bai~(2011), states that for every graph $G$ of order $n$
and every $1\le k\le n$,
$\sum_{i=1}^k\lambda_i(G)\le\sum_{i=1}^k d_i^*(G)$,
where $\lambda_1\ge\cdots\ge\lambda_n$ are the Laplacian eigenvalues and
$d_1^*\ge\cdots\ge d_n^*$ is the conjugate degree sequence.
In this paper we determine exactly when equality holds.
Using the split-graph trace inequality developed by Kothari and
Tudose~(2026) in their proof of Brouwer's Laplacian conjecture---which
relies on Bai's theorem and also establishes the equivalence between the
two conjectures---together with the recent characterization of the Brouwer
equality cases by Cai, Chen, Yang and Zhang~(2027), we prove that equality
holds in the Grone--Merris inequality if and only if the graph $G$ belongs to
one of two explicitly described families.
Both families are obtained from a threshold graph by a surgical operation
at one terminal block: in the first family, edges are removed from the
initial dominating block; in the second, edges are added inside the initial
isolated block.
Our analysis yields a complete combinatorial description of all pairs $(G,k)$
for which the Grone--Merris bound is tight.
\end{abstract}

\section{Introduction}
Throughout this paper, all graphs are finite, simple and undirected.
Let \(G\) be a  simple  graph with vertex set \(V(G)\) and edge set \(E(G)\), denoted by
\(
n=|V(G)|
\)
and size
\(
m=e(G)=|E(G)|.
\) 
The Laplacian matrix of $G$ is $L(G)=D(G)-A(G)$, where $D(G)$ is the diagonal
degree matrix and $A(G)$ is the adjacency matrix.
Denote the Laplacian eigenvalues of $G$ by
\[
\lambda_1(G)\ge\lambda_2(G)\ge\cdots\ge\lambda_{n-1}(G)\ge\lambda_n(G)=0.
\]
For a degree sequence $d_1(G)\ge d_2(G)\ge\cdots\ge d_n(G)$ of $G$,
the \emph{conjugate degree sequence} $d_1^*(G)\ge d_2^*(G)\ge\cdots\ge d_n^*(G)$ is
defined by
\[
d_k^*(G)=|\{v\in V:d_G(v)\ge k\}|,\qquad 1\le k\le n.
\]
For $1\le k\le n-1$, an equivalent and often useful expression is
\[
\sum_{i=1}^k d_i^*(G)=\min_{0\le r\le n}\Bigl(rk+\sum_{i=r+1}^n d_i(G)\Bigr),
\]
where the minimum is attained precisely for integers $r$ satisfying
$d_{k+1}^*(G)\le r\le d_k^*(G)$.
When no confusion may arise, we simply write \(\lambda_i\), \(d_k\) and \(d_k^*\).
 Denote by \(\mathbf 1\) the all-one column vector of appropriate size, $I$ the identity matrix of appropriate size, and $J$ the all-one matrix of of appropriate size.
The Grone--Merris conjecture, proved by Bai and now known as the Grone--Merris--Bai theorem, is stated as follows.
\begin{theorem}\cite{GM1994,Bai2011}
Let $G$ be a simple connected graph with the Laplacian eigenvalues $\lambda_1\ge\ldots\lambda_n$, the conjugate degree sequence $ d_1^*\ge\ldots\ge d_n^*$. Then
\begin{equation}\label{gm1}
(\lambda_1(G),\ldots,\lambda_n(G))\preceq (d_1^*(G),\ldots,d_n^*(G)).
\end{equation}
Equivalently,
\begin{equation}\label{GM}
\sum_{i=1}^k \lambda_i(G)\leq \sum_{i=1}^k d_i^*(G),\ \ \
\qquad  \text{for } \ k=1, \ldots, n.
\end{equation}
\end{theorem}
(For $k=n$ both sides equal $2e(G)$, so the inequality is trivial.)
A related problem is Brouwer's Laplacian conjecture, which states that for
every graph $G$ of order $n$ and every $1\le k\le n$,
\begin{equation}\label{Brouwer}
\sum_{i=1}^k\lambda_i(G)\le e(G)+\binom{k+1}{2}.
\end{equation}
(Again $k=n$ is trivial.)
This conjecture, proposed by Brouwer~\cite{BrouwerHaemers}, attracted sustained
attention and was recently proved by Kothari and Tudose~\cite{Kothari2026}.
The central innovation of~\cite{Kothari2026} is a split-graph trace inequality
(Lemma~3.3 of that paper):
for a split graph $H$ with clique $K$ and independent set $S$, and any suitable
orthogonal projection $Q$, one has
$\operatorname{tr}(Q(L_H-rI))\le e_H(K,S)$.
Kothari and Tudose proved Brouwer's conjecture by applying this inequality
together with the Grone--Merris--Bai theorem restricted to split graphs.
They also proved the reverse implication: if Brouwer's conjecture holds for all
graphs, then the Grone--Merris--Bai theorem holds for all split graphs.
In this sense, the two statements are equivalent.

Recently, Cai, Chen, Yang and Zhang~\cite{Cai2026} characterized the equality
cases of Brouwer's conjecture, showing that~\eqref{Brouwer} is an equality if
and only if $G$ is a threshold graph with clique number $k+1$.
The detailed spectral analysis of threshold graph Laplacians plays a crucial
role in the present work.

Despite these advances, the question of exactly which pairs $(G,k)$ achieve
equality in the Grone--Merris inequality~\eqref{GM} has remained open.
In this paper we settle this question completely.
Our main result, Theorem~\ref{mainthm}, shows that equality holds if and only
if $G$ belongs to one of two explicitly described families, both obtained by
modifying a single terminal block of a threshold graph.
\begin{itemize}
\item \textbf{Type~I (lower terminal clique-block replacement):}
$G$ is obtained from a threshold graph by deleting arbitrary edges from the
first dominating block, provided the resulting graph $F$ satisfies a
minimum-degree and component condition relative to $k$.
\item \textbf{Type~II (upper terminal independent-block replacement):}
$G$ is obtained from a threshold graph, by adding
arbitrary edges inside the first isolated block, subject to analogous
constraints.
\end{itemize}

Our proof follows the structure of~\cite{Kothari2026}: we unwind each
inequality in the derivation of~\eqref{GM} via the split-graph trace
inequality, translate each slackness condition into a geometric constraint on
the projection $Q$, and solve these constraints simultaneously using the
spectral theory of threshold graphs developed in~\cite{Cai2026} and the
detailed analysis of the split-graph trace inequality we carry out in
Section~\ref{sec:split-trace}.

The paper is organized as follows.
Section~\ref{sec:prelim} collects notation and presents the two equivalent
definitions of threshold graphs that we shall need.
Section~\ref{sec:thg-spectrum} gives a self-contained description of the
Laplacian spectrum of a threshold graph.
Section~\ref{sec:split-trace} analyzes the equality conditions of the
split-graph trace inequality.
Section~\ref{sec:reduction} reduces the GM equality to five explicit
conditions.
Section~\ref{sec:main} states and proves the main theorem.

\section{Preliminaries}\label{sec:prelim}

\subsection{Notation}

Throughout, $G=(V,E)$ is a simple graph of order $n$.
We write $|V|=n$ and $|E|=e(G)$.
For a vertex $v$, $d_G(v)$ denotes its degree and $N_G(v)$ its neighbourhood.
For a subset $X\subseteq V$, $G[X]$ denotes the induced subgraph and
$e(G[X])$ the number of edges inside $X$.
For disjoint subsets $X,Y\subseteq V$, $e_G(X,Y)$ is the number of edges
with one endpoint in $X$ and the other in $Y$.

For a real number $\lambda$, let
\[
E_\lambda(L(G))=\{x\in\mathbb R^V: L(G)x=\lambda x\}
\]
denote the $\lambda$-eigenspace. For a subspace $L\subset \mathbb R^n$, an $n\times n$ real matrix $P$ is called an orthogonal projection onto $L$ if $P^2=P$, $P^T=P$, $ImP=L$, which is uniquely determined. 
For a real number $r>0$, set
\[
E_{>r}(L(G))=\bigoplus_{\lambda>r}E_\lambda(L(G)),\qquad
E_{<r}(L(G))=\Bigl(\bigoplus_{\lambda<r}E_\lambda(L(G))\Bigr)\cap\mathbf 1^\perp.
\]
%(In words, $E_{<r}$ consists of all eigenvectors with eigenvalue $<r$ that are orthogonal to $\mathbf 1$; this excludes the all-ones direction while retaining any other zero-eigenvalue directions.)
and write
\[
Q_{r}=P_{(r,\infty)}(L(G))
\]
for the orthogonal projection onto $E_{>r}(L(G))$.

For a subspace $U\subseteq\mathbb R^V$, $P_U$ denotes the orthogonal
projection onto $U$.
For a subset $X\subseteq V$, $I_X$ is the $|X|\times|X|$ identity matrix and
$J_X$ the all-ones matrix of the same size; both are extended to operators on
$\mathbb R^V$ by acting as zero outside $X$.

\subsection{Threshold graphs}

Threshold graphs were introduced by Chv\'atal and Hammer~\cite{ChvatalHammer1977}
and are the subject of the monograph~\cite{MahadevPeled}.
They admit many equivalent characterizations; we shall make use of the
following two.

\begin{definition}[Creation sequence]\label{def:creation}
A graph $H$ is a \emph{threshold graph} if it can be constructed from a
single vertex by repeatedly applying the following two operations:
\begin{itemize}
\item add a dominating vertex (adjacent to all existing vertices);
\item add an isolated vertex (adjacent to none of the existing vertices).
\end{itemize}
Grouping consecutive operations of the same type, every threshold graph
admits a unique \emph{creation sequence} of the form
\[
U_1,D_1,U_2,D_2,\ldots,U_m,D_m,
\]
where each $U_i$ is a block of consecutive dominating vertices and each $D_i$
is a block of consecutive isolated vertices.
Only $D_m$ may be empty; all other blocks are nonempty.
\end{definition}

With this convention,
\[
K:=U_1\cup U_2\cup\cdots\cup U_m
\]
is a maximum clique of $H$, and
\[
S:=D_1\cup D_2\cup\cdots\cup D_m
\]
is an independent set.
The clique number is $\kappa:=\sum_{i=1}^m|U_i|$.
The edge relations are completely determined by the blocks:
all pairs of distinct vertices in $K$ are adjacent;
no two vertices in $S$ are adjacent;
$D_i\sim U_j$ if and only if $j>i$.
Consequently,
\[
d_H(v)=\sum_{j=i+1}^m|U_j|\;\;(v\in D_i),\qquad
d_H(v)=\kappa-1+\sum_{j=1}^{i-1}|D_j|\;\;(v\in U_i).
\]

\begin{theorem}[Mahadev--Peled~\cite{MahadevPeled}]\label{thm:thg-equiv}
For a graph $G$, the following are equivalent:
\begin{enumerate}
\item[(i)] $G$ is a threshold graph (can be built by repeatedly adding
dominating or isolated vertices);
\item[(ii)] $G$ is a split graph whose vertex set can be partitioned into a
clique $K$ and an independent set $S$ such that the neighborhoods of the
vertices of $S$ into $K$ are nested, i.e.,~totally ordered by inclusion;
\item[(iii)] $G$ contains no induced $P_4$, $C_4$, or $2K_2$;
\item[(iv)] there exist nonnegative real numbers $w_v$ ($v\in V$) and a
threshold $t$ such that $uv\in E(G)$ if and only if $w_u+w_v>t$.
\end{enumerate}
\end{theorem}

In this paper we shall primarily use the creation sequence viewpoint
(Definition~\ref{def:creation}), as it gives direct access to the block
structure needed for the spectral analysis.
The nested-neighborhood condition~(ii) is used implicitly whenever we refer
to a threshold graph as a split graph with a particular $K$--$S$ adjacency
pattern.

\subsection{Brouwer's conjecture and its equality cases}

Brouwer's Laplacian conjecture~\eqref{Brouwer} was proved by Kothari and
Tudose~\cite{Kothari2026}.
The equality cases were recently determined by Cai, Chen, Yang and
Zhang~\cite{Cai2026}:
\begin{theorem}[~\cite{Cai2026}]\label{thm:brouwer-eq}
Let \(G\) be a graph.  For $1\leq k\leq |V(G)|$,
\begin{equation}\label{equforBC}
\sum_{i=1}^k\lambda_i(G) = e(G) + \binom{k+1}{2} %, \quad \text{for}\  k=1,\ldots,n-1.
\end{equation}
if and only if \(G\) is a threshold graph with clique number
\(k+1\).
\end{theorem}
This result pins down the structure of the split graph $H$ that appears in the
split-graph trace inequality when the considered inequality is tight.

\section{Laplacian Spectrum of Threshold Graphs}\label{sec:thg-spectrum}

In this section we give a self-contained description of all Laplacian
eigenvectors and eigenvalues of a threshold graph $H$ whose creation
sequence is
\[
U_1,D_1,U_2,D_2,\ldots,U_m,D_m\qquad(\text{only }D_m\text{ may be empty}).
\]
Set $\kappa=\sum_{i=1}^m|U_i|$.

\begin{proposition}[Laplacian eigenvectors of a threshold graph]\label{prop:thg-spec}
The Laplacian eigenvalues and eigenvectors of $H$ are as follows.

\medskip\noindent\textbf{Type~I: difference vectors inside $D_i$-blocks.}
For each $i$ with $|D_i|>0$, every vector $x$ satisfying
$\operatorname{supp}(x)\subseteq D_i$ and $\sum_{v\in D_i}x_v=0$
is an eigenvector with eigenvalue
$\sum_{j=i+1}^m|U_j|$.
This gives $|D_i|-1$ linearly independent eigenvectors.

\medskip\noindent\textbf{Type~II: difference vectors inside $U_i$-blocks.}
For each $i$, every vector $x$ satisfying
$\operatorname{supp}(x)\subseteq U_i$ and $\sum_{v\in U_i}x_v=0$
is an eigenvector with eigenvalue
$\kappa+\sum_{j=1}^{i-1}|D_j|$.
This gives $|U_i|-1$ linearly independent eigenvectors.

\medskip\noindent\textbf{Type~III: block-constant eigenvectors of $D$-type.}
For each $h=1,\ldots,m$ with $|D_h|>0$, define
$\psi^{D,h}$ by setting it equal to $1$ on $U_1,\ldots,U_h$ and
$D_1,\ldots,D_{h-1}$,
equal to
$-\frac{\sum_{i=1}^h|U_i|+\sum_{i=1}^{h-1}|D_i|}{|D_h|}$
on $D_h$, and $0$ on all remaining blocks.
Then
$L(H)\,\psi^{D,h}=(\sum_{j=h+1}^m|U_j|)\,\psi^{D,h}$.

\medskip\noindent\textbf{Type~IV: block-constant eigenvectors of $U$-type.}
For each $p=2,\ldots,m$, define
$\psi^{U,p}$ by setting it equal to $1$ on $U_1,\ldots,U_{p-1}$ and
$D_1,\ldots,D_{p-1}$,
equal to
$-\frac{\sum_{i=1}^{p-1}|U_i|+\sum_{i=1}^{p-1}|D_i|}{|U_p|}$
on $U_p$, and $0$ on all remaining blocks.
Then
$L(H)\,\psi^{U,p}=(\kappa+\sum_{j=1}^{p-1}|D_j|)\,\psi^{U,p}$.

Finally, $\mathbf 1$ is an eigenvector with eigenvalue $0$.
\end{proposition}

\begin{proof}
We verify each type using $(Lx)_v=\sum_{u\sim v}(x_v-x_u)$.

\medskip\noindent\textit{Type~I.}
Let $x$ be supported on $D_i$ with zero sum.
If $v\in D_i$, then $v$ is adjacent exactly to $U_{i+1}\cup\cdots\cup U_m$,
on which $x$ vanishes, so $(Lx)_v=(\sum_{j=i+1}^m|U_j|)\,x_v$.
If $v\notin D_i$, then $v$ is adjacent either to all vertices of $D_i$ or to
none, giving a contribution of $-\sum_{u\in D_i}x_u=0$ or $0$ respectively.

\medskip\noindent\textit{Type~II.}
Let $x$ be supported on $U_i$ with zero sum.
If $v\in U_i$, then $(Lx)_v=d_H(v)x_v-\sum_{u\in U_i\setminus\{v\}}x_u$.
Since $\sum_{u\in U_i}x_u=0$, the sum over $U_i\setminus\{v\}$ equals $-x_v$,
so $(Lx)_v=(d_H(v)+1)x_v$.
For $v\in U_i$, $d_H(v)=\kappa-1+\sum_{j=1}^{i-1}|D_j|$, hence the eigenvalue
is $\kappa+\sum_{j=1}^{i-1}|D_j|$.
If $v\notin U_i$, the contribution vanishes as in Type~I.

\medskip\noindent\textit{Types~III and IV.}
Let $x$ be constant on each block; write $x(D_i)=\alpha_i$, $x(U_i)=\beta_i$.
For $v\in D_i$,
$(Lx)_v=\sum_{j=i+1}^m|U_j|(\alpha_i-\beta_j)$.
For $v\in U_i$,
$(Lx)_v=\sum_{j\neq i}|U_j|(\beta_i-\beta_j)+\sum_{j=1}^{i-1}|D_j|(\beta_i-\alpha_j)$.

For Type~III, fix $h$ with $|D_h|>0$.
Set $\beta_i=1$ for $i\le h$, $\alpha_i=1$ for $i<h$,
$\alpha_h=-\frac{\sum_{i=1}^h|U_i|+\sum_{i=1}^{h-1}|D_i|}{|D_h|}$,
and $\alpha_i=\beta_i=0$ for $i>h$.

If $v\in D_i$ with $i<h$:
$(Lx)_v=\sum_{j=h+1}^m|U_j|$, matching the claimed eigenvalue.
If $v\in D_h$: $(Lx)_v=(\sum_{j=h+1}^m|U_j|)\alpha_h$.
If $v\in D_i$ ($i>h$): all neighboring $U$-blocks have value $0$.
If $v\in U_i$ with $i\le h$:
$(Lx)_v=\sum_{j=h+1}^m|U_j|$.
If $v\in U_i$ with $i>h$: $\beta_i=0$, and
$(Lx)_v=-\sum_{i=1}^h|U_i|-\sum_{i=1}^{h-1}|D_i|-|D_h|\alpha_h=0$.
Thus $L\psi^{D,h}=(\sum_{j=h+1}^m|U_j|)\,\psi^{D,h}$.

For Type~IV, fix $p\in\{2,\ldots,m\}$.
Set $\beta_i=1$, $\alpha_i=1$ for $i<p$,
$\beta_p=-\frac{\sum_{i=1}^{p-1}|U_i|+\sum_{i=1}^{p-1}|D_i|}{|U_p|}$,
$\alpha_p=0$, and $\alpha_i=\beta_i=0$ for $i>p$.

For $v\in D_i$ ($i<p$):
$(Lx)_v=|U_p|(1-\beta_p)+\sum_{j=p+1}^m|U_j|
=\kappa+\sum_{j=1}^{p-1}|D_j|$.
For $v\in U_i$ ($i<p$): the same computation applies.
For $v\in U_p$:
\[
(Lx)_v=(\sum_{i=1}^{p-1}|U_i|)(\beta_p-1)
+(\sum_{j=p+1}^m|U_j|)\beta_p
+(\sum_{j=1}^{p-1}|D_j|)(\beta_p-1)
=(\kappa+\sum_{j=1}^{p-1}|D_j|)\,\beta_p.
\]
For $v\in U_i$ ($i>p$): all contributions cancel by the definition of $\beta_p$.

These eigenvectors, together with $\mathbf 1$, span $\mathbb R^V$, giving the
full spectral decomposition.
\end{proof}

\begin{remark}\label{rem:spectral-facts}
The following consequences of Proposition~\ref{prop:thg-spec} are essential
for the main proof.
Let $r_0=\kappa,\ K=\bigcup_{i=1}^mU_i,\ S=\bigcup_{i=1}^mD_i$.

\begin{enumerate}
\item[(a)] The eigenspace $E_{r_0}(L_H)=Z_1\subset \mathcal K$, where
$\mathcal K=\{x:x|_S=0,\;\langle x,\mathbf 1_K\rangle=0\}$.

\item[(b)] The eigenvalues larger than $r_0$ are the $U$-type eigenvalues:
$E_{>r_0}(L_H)=\bigoplus_{p=2}^m\bigl(Z_p\oplus\operatorname{span}\{\psi^{U,p}\}\bigr)$,
where $Z_p=\{x:\operatorname{supp}(x)\subseteq U_p,\;\sum_{u\in U_p}x_u=0\}$.
In particular,
$\dim E_{>r_0}(L_H)=\sum_{p=2}^m|U_p|=r_0-|U_1|$.

\item[(c)] The eigenvalues smaller than $r_0$ are the $D$-type eigenvalues:
$E_{<r_0}(L_H)=\bigoplus_{i=1,|D_i|>0}^m\bigl(Y_i\oplus\operatorname{span}\{\psi^{D,i}\}\bigr)$,
where $Y_i=\{x:\operatorname{supp}(x)\subseteq D_i,\;\sum_{v\in D_i}x_v=0\}$.
Hence $\dim E_{<r_0}(L_H)=\sum_{i=1}^m|D_i|=n-r_0$.

\item[(d)] The projection onto $E_{>r_0}(L_H)$ admits the explicit formula
\[
P_{E_{>r_0}(L_H)}=\sum_{p=2}^m\Bigl(I_{U_p}-\frac1{|U_p|}J_{U_p}\Bigr)
+\sum_{p=2}^m\frac{\psi^{U,p}(\psi^{U,p})^T}{\|\psi^{U,p}\|^2}.
\]
\end{enumerate}
\end{remark}

\section{Equality in the Split-Graph Trace Inequality}\label{sec:split-trace}

We now give a detailed analysis of the equality conditions of
Lemma~3.3 from~\cite{Kothari2026}.

\begin{theorem}[Lemma~3.3 of~\cite{Kothari2026}]\label{thm:split-trace}
Let $H$ be a split graph of order $n$ with clique $K$ ($|K|=r_0$) and
independent set $S$.
Let $Q$ be an orthogonal projection on $\mathbb R^{V(H)}$ with
$Q\mathbf 1=0$, and set $q=\operatorname{rank}Q$.
Then
\[
\operatorname{tr}\bigl(Q(L_H-r_0I_n)\bigr)\le e_H(K,S),
\]
and equality holds if and only if exactly one of the following occurs.

\medskip\noindent\textbf{Case~1: $q<r_0$.}
$H$ is a threshold graph with clique number $r_0$, whose creation sequence is
$U_1,D_1,\ldots,U_m,D_m$ (only $D_m$ possibly empty), with
$K=U_1\cup\cdots\cup U_m$.
Moreover,
\[
Q=P_{E_{>r_0}(L_H)}+P_W
\]
for some subspace $W\le Z_1$, where
$Z_1=\{x:\operatorname{supp}(x)\subseteq U_1,\;\sum_{u\in U_1}x_u=0\}$.

\medskip\noindent\textbf{Case~2: $q\ge r_0$.}
$\overline H$ (the complement of $H$) is a threshold graph with clique
number $n-r_0$ whose clique side is $S$.
Equivalently, there exist partitions
$S=B_1\cup\cdots\cup B_\ell$, $K=A_1\cup\cdots\cup A_\ell$
(only $A_\ell$ possibly empty) such that the cross-edges are
$A_i\sim B_j\iff j\le i$.
Let $\widetilde Z_1=\{x:\operatorname{supp}(x)\subseteq B_1,\;
\sum_{v\in B_1}x_v=0\}$.
Then
\[
Q=I_n-P_{\mathbf 1}-P_{E_{<r_0}(L_H)}-P_{\widetilde W}
\]
for some subspace $\widetilde W\le\widetilde Z_1$.
\end{theorem}

\begin{proof}
We analyze the two cases separately.

\medskip\noindent\textit{Case $q<r_0$.}
Assume $\operatorname{tr}(Q(L_H-r_0I_n))=e_H(K,S)$.
Define $\mathcal K=\{x:x|_S=0,\;\langle x,\mathbf 1_K\rangle=0\}$
($\dim\mathcal K=r_0-1$).
Since $\operatorname{Im}Q$ has dimension $q$ and is contained in
$\mathbb R^V$, we have
\[
\dim\bigl(\mathcal K\cap(\operatorname{Im}Q)^\perp\bigr)
\ge\dim\mathcal K-\dim(\operatorname{Im}Q)
=(r_0-1)-q.
\]
Since $q<r_0$, we have $(r_0-1)-q\ge 0$.
Hence we may pick a subspace
$L\subseteq\mathcal K\cap(\operatorname{Im}Q)^\perp$
of dimension $r_0-1-q$.
Let $Q_L$ be the orthogonal projection onto $L$.
Then $Q+Q_L$ is an orthogonal projection of rank $r_0-1$.

We estimate
\begin{align}
\operatorname{tr}(Q(L_H-r_0I_n))\notag
&=\operatorname{tr}((Q+Q_L)L_H)-\operatorname{tr}(Q_LL_H)-r_0q\\
&\le\sum_{i=1}^{r_0-1}\lambda_i(H)-\operatorname{tr}(Q_LL_H)-r_0q \label{di1}\\
&\le\sum_{i=1}^{r_0-1}\lambda_i(H)-r_0(r_0-1-q)-r_0q \label{di2}\\
&\le e(H)+\binom{r_0}{2}-r_0(r_0-1) \label{di3}\\
&=e_H(K,S).\notag
\end{align}

Inequality~(\ref{di1}) uses that $Q+Q_L$ is an orthogonal projection of rank $r_0-1$, so by Courant-Fischer theorem the trace is
at most the sum of the $r_0-1$ largest Laplacian eigenvalues of $H$.
Equality requires $Q+Q_L$ to project onto the eigenspace of the top $r_0-1$
eigenvalues.

For~(\ref{di2}), take any $v\in L\subseteq\mathcal K$.
Since $v|_S=0$ and $H[K]$ is a complete graph,
\[
\langle v,L_Hv\rangle
=\sum_{ij\in E(H)}(v_i-v_j)^2
\ge\sum_{ij\in E(K)}(v_i-v_j)^2
=r_0\|v\|^2,
\]
using that $\sum_{u\in K}v_u=0$ and $H[K]$ is complete.
Summing over an orthonormal basis of $L$,
$\operatorname{tr}(Q_LL_H)\ge r_0(r_0-1-q)$.
Equality in~(\ref{di2}) requires $\langle v,L_Hv\rangle=r_0\|v\|^2$ for all $v\in L$,
hence $(v_i-v_j)^2=0$ for every $ij\in E(K,S)$.
Since $v|_S=0$, this forces $v_i=0$ whenever $i\in K$ has an $S$-neighbor.

Now compute $L_Hv$ for $v\in L$:
\[
L_Hv=\sum_{ij\in E(H)}(v_i-v_j)(e_i-e_j)
=\sum_{ij\in E(K)}(v_i-v_j)(e_i-e_j)
=r_0v.
\]
Thus $L\subseteq E_{r_0}(L_H)$, so $L\subseteq E_{r_0}(L_H)\cap\mathcal K$.

Inequality~(\ref{di3}) is the Brouwer inequality, with equality if and only if $H$ is a threshold graph
with clique number $r_0$, by Theorem~\ref{thm:brouwer-eq}.

Since $Q+Q_L$ projects onto the top $r_0-1$ eigenspace and
$L\subseteq E_{r_0}(L_H)\cap\mathcal K=Z_1$, we obtain
$\operatorname{Im}Q=E_{>r_0}(L_H)\oplus W$ with $W=Z_1\cap(\operatorname{Im}Q)$.

\medskip\noindent\textit{Case $q\ge r_0$.}
Following~\cite{Kothari2026}, one observes the duality
\[
\operatorname{tr}(Q(L_H-r_0I_n))=e_H(K,S)
\iff
\operatorname{tr}(\overline Q(L_{\overline H}-(n-r_0)I_n))=e_{\overline H}(K,S),
\]
where $\overline Q=I_n-P_{\mathbf 1}-Q$.
Moreover, $\operatorname{rank}\overline Q=n-1-q<n-r_0$.
Applying the Case~1 analysis to $\overline H$ and $\overline Q$, we conclude
that $\overline H$ is a threshold graph with clique number $n-r_0$.
Translating back to $H$ yields the claimed partitions $A_i$, $B_i$ and
$Q=I_n-P_{\mathbf 1}-P_{E_{<r_0}(L_H)}-P_{\widetilde W}$.

\medskip\noindent
The converse direction (structure $\Rightarrow$ equality) is verified by
direct computation.
\end{proof}

\section{Reduction of the Grone--Merris Equality}\label{sec:reduction}

We now unwind the proof of the GM inequality via the split-graph trace
inequality.

\begin{lemma}\label{lem:min-equiv}
Let $\{a_r\}_{r=0}^n$, $\{b_r\}_{r=0}^n$, $\{c_r\}_{r=0}^n$ be real sequences
with $a_r\le b_r\le c_r$ for all $r$.
Then $\min_r a_r=\min_r b_r=\min_r c_r$
if and only if there exists $r_0$ such that
$a_{r_0}=b_{r_0}=c_{r_0}=\min_r a_r=\min_r b_r=\min_r c_r$.
Equivalently, for every $r_0$ with $c_{r_0}=\min_r c_r$, we have
$a_{r_0}=b_{r_0}=c_{r_0}$.
\end{lemma}

\begin{proof}
The ``if'' direction is immediate.
Conversely, let $m$ be the common minimum.
Pick $r_0$ with $c_{r_0}=m$.
Then $m\le a_{r_0}\le b_{r_0}\le c_{r_0}=m$, forcing equality throughout.
\end{proof}

Recall the chain of inequalities from~\cite{Kothari2026}:
\begin{align*}
\sum_{i=1}^k\lambda_i(G)
&=\min_{0\le r\le n}\Bigl(rk+\sum_{i=1}^k(\lambda_i-r)_+\Bigr) \\
&\le\min_{0\le r\le n}\Bigl(rk+\sum_{i=1}^n(\lambda_i-r)_+\Bigr) \\
&=\min_{0\le r\le n}\bigl(rk+\operatorname{tr}(Q_{r}(L_G-rI_n))\bigr)\\
&\le\min_{0\le r\le n}\Bigl(rk+\sum_{i=r+1}^n d_i(G)\Bigr) \\
&=\sum_{i=1}^k d_i^*(G).
\end{align*}
For each integer $r$, set
\[
A_r=rk+\sum_{i=1}^k(\lambda_i-r)_+,\quad
B_r=rk+\operatorname{tr}(Q_{r}(L_G-rI_n)),\quad
C_r=rk+\sum_{i=r+1}^n d_i(G).
\]
Then $A_r\le B_r\le C_r$ for all $r$, and Lemma~\ref{lem:min-equiv} applies.

\begin{proposition}\label{prop:five-conditions}
Let $G$ be a graph of order $n$ and $1\le k\le n-1$.
(The case $k=n$ is trivial: both sides of~\eqref{GM} equal $2e(G)$.)
The Grone--Merris inequality~\eqref{GM} is an equality if and only if there
exists an integer $r_0$ such that:
\begin{enumerate}
\item[(i)] $\lceil\lambda_{k+1}(G)\rceil\le r_0\le\lfloor\lambda_k(G)\rfloor$;
\item[(ii)] $d_{k+1}^*(G)\le r_0\le d_k^*(G)$;
\item[(iii)] for every $ij\in E(G[S])$, $Q_{r_0}(e_i-e_j)=e_i-e_j$;
\item[(iv)] for every $i,j\in K$ with $ij\notin E(G)$, $Q_{r_0}(e_i-e_j)=0$;
\item[(v)] $\operatorname{tr}(Q_{r_0}(L_H-r_0I_n))=e_H(K,S)$,
\end{enumerate}
where $K$ is any set of $r_0$ vertices of largest degree,
$S=V\setminus K$, $H$ is the split graph obtained from $G$ by making $K$ a
clique and $S$ an independent set, and $Q_{r_0}=P_{(r_0,\infty)}(L_G)$.
\end{proposition}

\begin{proof}
By Lemma~\ref{lem:min-equiv}, $\sum_{i=1}^k\lambda_i=\sum_{i=1}^k d_i^*$ iff
$\min A_r=\min B_r=\min C_r$, which occurs iff there exists $r_0$ with
$A_{r_0}=B_{r_0}=C_{r_0}$ and all three equal the common minimum.

Condition~(i) is obtained as follows.  The function $A_r=rk+\sum_{i=1}^k(\lambda_i-r)_+$ is constant for $r\le\lambda_k$ and strictly increasing for $r>\lambda_k$; hence $\min_r A_r=\sum_{i=1}^k\lambda_i$ is attained for all integers $0\le r\le\lfloor\lambda_k\rfloor$.  The equality $A_{r_0}=\min A_r$ therefore gives $r_0\le\lfloor\lambda_k\rfloor$.  On the other hand, the equality $A_{r_0}=B_{r_0}$ is equivalent to $\sum_{i=k+1}^n(\lambda_i-r_0)_+=0$, which forces $r_0\ge\lambda_{k+1}$, and with $r_0$ integer, $r_0\ge\lceil\lambda_{k+1}\rceil$.
Condition~(ii) follows because the minimum of $C_r$ is achieved exactly for
$r\in[d_{k+1}^*,d_k^*]$ (this holds for $k\le n-1$).

Now $B_{r_0}=C_{r_0}$ means
$\operatorname{tr}(Q_{r_0}(L_G-r_0I_n))=\sum_{i=r_0+1}^n d_i(G)$.
Expanding,
\begin{align*}
\operatorname{tr}(Q_{r_0}(L_G-r_0I_n))
&=\sum_{ij\in E(G)}\|Q_{r_0}(e_i-e_j)\|^2-r_0q\\
&\le 2e(G[S])+\sum_{ij\in E(H)}\|Q_{r_0}(e_i-e_j)\|^2-r_0q\\
&=\operatorname{tr}(Q_{r_0}(L_H-r_0I_n))+2e(G[S])\\
&\le e_H(K,S)+2e(G[S])=\sum_{i=r_0+1}^n d_i(G),
\end{align*}
where $q=\operatorname{rank}Q_{r_0}$.
Equality in the first inequality requires
$\|Q_{r_0}(e_i-e_j)\|^2=1$ for every $ij\in E(G[S])$, which is~(iii),
and $\|Q_{r_0}(e_i-e_j)\|^2=0$ for every non-edge in $G[K]$, which is~(iv).
Equality in the second inequality is precisely~(v).
\end{proof}

\section{Characterization of the Equality Cases}\label{sec:main}

\subsection{Statement of the main theorem}

\begin{theorem}[Equality cases of Grone--Merris]\label{mainthm}
Let $G$ be a graph, and $1\le k\le |V(G)|-1$.
Equality $\sum_{i=1}^k\lambda_i(G)=\sum_{i=1}^k d_i^*(G)$ holds if and only
if $(G,k)$ belongs to exactly one of the following two families.

\medskip\noindent\textbf{Type~I (lower terminal clique-block replacement).}
There exists a threshold graph $H$ with creation sequence
$U_1,D_1,U_2,D_2,\ldots,U_m,D_m$ (only $D_m$ can be empty),
such that $G$ is obtained from $H$ by replacing the clique $H[U_1]$ with an
arbitrary graph $F$ on $U_1$, while keeping all other edges of $H$ unchanged.
And $k$ satisfies
\begin{equation}\label{eq:typeI-range}
\sum_{i=2}^m|U_i|\;\le\; k\;\le\;
\sum_{i=2}^m|U_i|+\min\bigl\{\delta(F),\,c(\overline F)-1\bigr\},
\end{equation}
where $\delta(F)$ is the minimum degree of $F$ and $c(\overline F)$ is the
number of connected components of the complement $\overline F$
(on the vertex set $U_1$).

\medskip\noindent\textbf{Type~II (upper terminal independent-block replacement).}
There exists a threshold graph $H$ whose complement $\overline H$ is a threshold
graph with creation sequence
$B_1,A_1,B_2,A_2,\ldots,B_\ell,A_\ell$ (only $A_\ell$ can be empty),
so that $H$ has clique $K=A_1\cup\cdots\cup A_\ell$, independent set
$S=B_1\cup\cdots\cup B_\ell$, and cross-edges $A_i\sim B_j\iff j\le i$.
The graph $G$ is obtained from $H$ by replacing the edgeless subgraph
$H[B_1]$ with an arbitrary graph $F$ on $B_1$, keeping all other edges of $H$.
Let $r_0=\sum_{i=1}^\ell|A_i|$.
Then $k$ satisfies
\begin{equation}\label{eq:typeII-range}
r_0+\max\bigl\{\Delta(F),\,|B_1|-c(F)\bigr\}\;\le\; k\;\le\;
r_0+|B_1|-1,
\end{equation}
where $\Delta(F)$ is the maximum degree and $c(F)$ the number of connected
components of $F$.
\end{theorem}

\subsection{Proof of necessity}

Assume $\sum_{i=1}^k\lambda_i(G)=\sum_{i=1}^k d_i^*(G)$.
By Proposition~\ref{prop:five-conditions}, there exists $r_0$ satisfying
(i)--(v).
Let $K$ be a set of $r_0$ vertices of largest degree (such a set exists by
condition~(ii)), $S=V\setminus K$, and $H$ the split graph obtained by
making $K$ a clique and $S$ an independent set.
Set $Q=Q_{r_0}=P_{(r_0,\infty)}(L_G)$, and let $q=\operatorname{rank}Q$.
Condition~(v) and Theorem~\ref{thm:split-trace} imply that exactly one of
two cases occurs.

\medskip\noindent\textit{Case A: $q<r_0$.}
Then $H$ is a threshold graph with creation sequence
$U_1,D_1,\ldots,U_m,D_m$, $K=U_1\cup\cdots\cup U_m$, $r_0=\sum_i|U_i|$, and
$Q=P_{E_{>r_0}(L_H)}+P_W$ with $W\le Z_1$.
(Recall $Z_1=\{x:\operatorname{supp}(x)\subseteq U_1,\;\sum_{u\in U_1}x_u=0\}$.)

We first prove that $G[S]$ is empty.
Since $H$ is a threshold graph, $H[S]$ is an independent set,
for any nonzero vector $x$ supported on $S$, we have
$\langle x,L_Hx\rangle=\sum_{s\in S}d_H(s)x_s^2\le(r_0-|U_1|)\|x\|^2\le(r_0-1)\|x\|^2$.
On the other hand, every nonzero vector in
$\operatorname{Im}Q=E_{>r_0}(L_H)\oplus W$ has Rayleigh quotient at least
$r_0$ (since $E_{>r_0}(L_H)$ consists of eigenvectors with eigenvalues $>r_0$,
and $W\le Z_1\subseteq E_{r_0}(L_H)$).
If there were an edge $ij\in E(G[S])$, condition~(iii) would give
$e_i-e_j\in\operatorname{Im}Q$, but $e_i-e_j$ is supported on $S$ and
nonzero, contradicting the Rayleigh quotient bound.
Hence $E(G[S])=\varnothing$.

Next we locate the missing edges inside $K$.
Let $i,j\in K$ with $ij\notin E(G)$.
By condition~(iv), $Q(e_i-e_j)=0$, so $e_i-e_j\perp\operatorname{Im}Q$.
In particular, $e_i-e_j\perp E_{>r_0}(L_H)$.
Write $i\in U_a$, $j\in U_b$.

\begin{itemize}
\item If $a=b\ge 2$, then $e_i-e_j\in Z_a\subseteq E_{>r_0}(L_H)$,
contradicting $e_i-e_j\perp E_{>r_0}(L_H)$.
\item If $a\neq b$, assume without loss of generality $a<b$.
The eigenvector $\psi^{U,b}$ lies in $E_{>r_0}(L_H)$, and
\[
\langle e_i-e_j,\psi^{U,b}\rangle
=(\psi^{U,b})_i-(\psi^{U,b})_j
=1+\frac{\sum_{t=1}^{b-1}|U_t|+\sum_{t=1}^{b-1}|D_t|}{|U_b|}\neq0,
\]
so $e_i-e_j$ is not orthogonal to $E_{>r_0}(L_H)$, contradiction.
\end{itemize}
The only remaining possibility is $a=b=1$.
Thus all missing edges of $G[K]$ lie within $U_1$, and we set
$F:=G[U_1]$.

Hence $G$ is obtained from $H$ by deleting exactly the edges of
$M:=\overline F$ inside $U_1$ (where $\overline F=K_{|U_1|}\setminus F$).

\medskip\noindent\textit{Spectral platform.}
We analyze the effect on the Laplacian spectrum.
Since $G=H-M$,
\[
L_G=L_H-L_M,
\]
with $L_M$ extended by zero outside $U_1$.

For every eigenvector $y$ of $L_H$ outside $Z_1$, one checks that
$L_My=0$: vectors in $Z_p$ ($p\ge 2$) vanish on $U_1$;
block-constant vectors $\psi^{U,p}$ are constant ($=1$) on $U_1$, so
$\langle e_i-e_j,\psi^{U,p}\rangle=0$ for all $ij\in E(M)\subseteq\binom{U_1}{2}$;
$D$-type eigenvectors are likewise annihilated by $L_M$.
Hence $L_Gy=L_Hy$ for such $y$, and the corresponding eigenvalues are unchanged.

It remains to analyze the action on $Z_1$.
For $x\in Z_1$, since $H[U_1]$ is a complete graph and $U_1$ has no neighbors
in $S$, we have $L_H|_ {Z_1}=r_0I_{Z_1}$.
Using $L_M=L_{\overline F}$ and the
complement relation $L_F+L_{\overline F}=|U_1|I_{U_1}-J_{U_1}$,(with $L_F,\ L_{\overline F}$ extended by zero outside $U_1$.)
\[
L_M|_{Z_1}
=|U_1|I_{Z_1}-L_F|_{Z_1}.
\]
Therefore,
\[
L_G|_{Z_1}
=L_H|_{Z_1}-L_M|_{Z_1}
=r_0I_{Z_1}-\bigl(|U_1|I_{Z_1}-L_F|_{Z_1})
=(r_0-|U_1|)I_{Z_1}+L_F|_{Z_1}.
\]

Let $\nu_1\ge\cdots\ge\nu_{|U_1|-1}$ be the eigenvalues of
$L_F|_{Z_1}$ (all Laplacian eigenvalues of $F$ except the
zero eigenvalue on $\mathbf 1_{U_1}$).
Then, as a multiset,
\[
\operatorname{Spec}(L_G)=
\Bigl(\operatorname{Spec}(L_H)\setminus\{\underbrace{r_0,\ldots,r_0}_{|U_1|-1}\}\Bigr)
\cup
\bigl\{|U_2|+\cdots+|U_m|+\nu_1,\ldots,|U_2|+\cdots+|U_m|+\nu_{|U_1|-1}\bigr\}.
\]
Since $\nu_i\le|U_1|$ for all $i$, the new eigenvalues on $Z_1$ are at most
$|U_2|+\cdots+|U_m|+|U_1|=r_0$, so no eigenvalue $>r_0$ is created.
Hence
\[
\dim E_{>r_0}(L_G)=\dim E_{>r_0}(L_H)=|U_2|+\cdots+|U_m|.
\]
The multiplicity of $r_0$ coming from $Z_1$ equals the number of $\nu_i$
equal to $|U_1|$.
Using $L_{\overline F}|_{Z_1}=|U_1|I-L_F|_{Z_1}$,
$\nu_i=|U_1|$ iff $\overline F$ has eigenvalue $0$ on $Z_1$,
i.e., there are $c(\overline F)-1$ such $\nu_i$.
Condition~(i) therefore gives
\[
|U_2|+\cdots+|U_m|\le k\le|U_2|+\cdots+|U_m|+c(\overline F)-1.
\]

\medskip\noindent\textit{Degree platform.}
Degrees in $G$:
\begin{itemize}
\item $v\in S$: $d_G(v)=d_H(v)\le r_0-|U_1|$, with equality exactly on $D_1$;
\item $u\in U_1$: $d_G(u)=r_0-|U_1|+d_F(u)
\ge r_0-|U_1|+\delta(F)$;
\item $u\in K\setminus U_1$: $d_G(u)\ge r_0$ (since each such vertex is adjacent
to all other $K$-vertices and to at least $D_1\neq\varnothing$ in $S$).
\end{itemize}
Thus $d_G(K\setminus U_1)\ge d_G(U_1)\ge d_G(S)$, showing that $K$ consists
of the $r_0$ vertices of largest degree.

Now $d_{k+1}^*\le r_0$ means that at most $r_0$ vertices have degree $\ge k+1$, i.e., every vertex in $S$ has degree $\le k$.
The maximum degree in $S$ is $r_0-|U_1|$, so we need
$r_0-|U_1|\le k$.
Conversely, $r_0\le d_k^*$ means at least $r_0$ vertices have degree $\ge k$, i.e., every vertex in $K$ has degree $\ge k$. The minimum degree in $K$ is $r_0-|U_1|+\delta(F)$, so we need $r_0-|U_1|+\delta(F)\geq k$.

Combining with the spectral constraint, and noting
$r_0-|U_1|=|U_2|+\cdots+|U_m|$, yields~\eqref{eq:typeI-range}.

\medskip\noindent\textit{Case B: $q\ge r_0$.}
By Theorem~\ref{thm:split-trace}, $\overline H$ is a threshold graph whose
creation sequence can be written as
$B_1,A_1,B_2,A_2,\ldots,B_\ell,A_\ell$ (only $A_\ell$ can be empty),
with $S=B_1\cup\cdots\cup B_\ell$, $K=A_1\cup\cdots\cup A_\ell$, and
$r_0=\sum_{i=1}^\ell|A_i|$.
The cross-edges of $H$ are $A_i\sim B_j\iff j\le i$.
Moreover,
$Q=I_n-P_{\mathbf 1}-P_{E_{<r_0}(L_H)}-P_{\widetilde W}$ with
$\widetilde W\le\widetilde Z_1$, where
$\widetilde Z_1=\{x:\operatorname{supp}(x)\subseteq B_1,\;
\sum_{v\in B_1}x_v=0\}$.

We first locate the edges inside $S$.
Condition~(iii) requires $Q(e_i-e_j)=e_i-e_j$ for all $ij\in E(G[S])$.
Since $Q=I-P_{\mathbf 1}-P_{E_{<r_0}(L_H)}-P_{\widetilde W}$, we have
$\operatorname{Im}Q=(\operatorname{span}\{\mathbf 1\}\oplus E_{<r_0}(L_H)
\oplus\widetilde W)^\perp$, and in particular
$e_i-e_j\perp E_{<r_0}(L_H)$.

Recall that $\overline H$ is a threshold graph whose creation sequence is
$B_1,A_1,\ldots,B_\ell,A_\ell$ ($B_i$ dominating, $A_i$ isolated).
Its $U$-type eigenvectors belong to $E_{>n-r_0}(L_{\overline H})$; via
$L_H=nI-J-L_{\overline H}$ on $\mathbf 1^\perp$ they become
$E_{<r_0}(L_H)$.
Explicitly, $E_{<r_0}(L_H)$ is spanned by the following vectors
(see Proposition~\ref{prop:thg-spec} applied to $\overline H$):
\begin{itemize}
\item difference vectors inside $B_p$ ($p\ge 2$),
  which are supported on $B_p$ and have zero sum;
\item block-constant vectors $\eta_p$ ($p\ge 2$), where $\eta_p$ equals $1$ on
  $B_1,A_1,\ldots,B_{p-1},A_{p-1}$,
  equals $-\frac{\sum_{t=1}^{p-1}|B_t|+\sum_{t=1}^{p-1}|A_t|}{|B_p|}$ on $B_p$,
  and $0$ elsewhere.
\end{itemize}

Write $i\in B_a$, $j\in B_b$.
\begin{itemize}
\item If $a=b\ge 2$, then $e_i-e_j$ is a difference vector in $B_a$, hence
  $e_i-e_j\in E_{<r_0}(L_H)$, contradicting $e_i-e_j\perp E_{<r_0}(L_H)$.
\item If $a\neq b$, assume $a<b$.  Then $\eta_b\in E_{<r_0}(L_H)$ and
  $\langle e_i-e_j,\eta_b\rangle=(\eta_b)_i-(\eta_b)_j
  =1+\frac{\sum_{t=1}^{b-1}|B_t|+\sum_{t=1}^{b-1}|A_t|}{|B_b|}\neq0$,
  so $e_i-e_j$ is not orthogonal to $E_{<r_0}(L_H)$, contradiction.
\end{itemize}
The only remaining possibility is $a=b=1$.
Thus all edges of $G[S]$ are confined to $B_1$; set $F:=G[B_1]$, and note
$G[B_p]=\varnothing$ for $p\ge 2$.

Next we show that $G[K]$ is complete.
Let $i,j\in K$ with $ij\notin E(G)$.
By condition~(iv), $Q(e_i-e_j)=0$, so
$e_i-e_j\perp\operatorname{Im}Q$.
Hence
$e_i-e_j\in\operatorname{span}\{\mathbf 1\}\oplus E_{<r_0}(L_H)\oplus\widetilde W$.
Since $e_i-e_j\perp\mathbf 1$ and $\widetilde W\le\widetilde Z_1$ is supported
on $B_1$ (hence orthogonal to $e_i-e_j$, which is supported on $K$),
we obtain $e_i-e_j\in E_{<r_0}(L_H)$.

Now compute the Rayleigh quotient of $e_i-e_j$ with respect to $L_H$.
Since $H[K]$ is a complete graph on $r_0$ vertices and
$\sum_{u\in K}(e_i-e_j)_u=0$,
\[
\langle e_i-e_j,L_H(e_i-e_j)\rangle
\ge\sum_{uv\in E(H[K])}\bigl((e_i-e_j)_u-(e_i-e_j)_v\bigr)^2
=r_0\|e_i-e_j\|^2,
\]
contradicting $e_i-e_j\in E_{<r_0}(L_H)$.
Thus no such $i,j$ exist; $G[K]$ is a complete graph.

\medskip\noindent\textit{Spectral platform (Case~B).}
Here $G=H+F$, so $L_G=L_H+L_F$, with $L_F$ extended by zero outside $B_1$.
Every eigenvector of $L_H$ outside $\widetilde Z_1$ is either constant on
$B_1$ or vanishes on $B_1$, hence is annihilated by $L_F$.
Thus the perturbation only affects $\widetilde Z_1$.

On $\widetilde Z_1$, by the defining cross-edge relation $A_i\sim B_j\iff j\le i$,
every vertex of $B_1$ is adjacent to all blocks $A_1,\ldots,A_\ell$, i.e.,~to
every vertex of $K$.
Thus $d_H(v)=r_0$ for all $v\in B_1$, and for $x\in\widetilde Z_1$,
$L_H|_{\widetilde Z_1}=r_0 I$.
Therefore,
\[
L_G|_{\widetilde Z_1}=r_0I_{\widetilde Z_1}+L_F|_{\widetilde Z_1}.
\]

Let $\mu_1\ge\cdots\ge\mu_{|B_1|-1}$ be the eigenvalues of
$L_F|_{\widetilde Z_1}$.
Then, as multisets,
\[
\operatorname{Spec}(L_G)=
\Bigl(\operatorname{Spec}(L_H)\setminus\{\underbrace{r_0,\ldots,r_0}_{|B_1|-1}\}\Bigr)
\cup\bigl\{r_0+\mu_1,\ldots,r_0+\mu_{|B_1|-1}\bigr\}.
\]
Eigenvalues $>r_0$ from $\widetilde Z_1$ correspond to $\mu_i>0$, i.e., to
the positive Laplacian eigenvalues of $F$ on $\widetilde Z_1$.
Their number is $|B_1|-c(F)$.
Together with the $r_0$ eigenvalues $>r_0$ inherited from $H$, we obtain
\[
\dim E_{>r_0}(L_G)=r_0+|B_1|-c(F).
\]
The eigenvalues equal to $r_0$ from $\widetilde Z_1$ correspond to
$\mu_i=0$, of which there are $c(F)-1$.
Condition~(i) therefore yields
\[
r_0+|B_1|-c(F)\le k\le r_0+|B_1|-c(F)+(c(F)-1)=r_0+|B_1|-1.
\]

\medskip\noindent\textit{Degree platform (Case~B).}
\begin{itemize}
\item $v\in B_1$: $d_G(v)=r_0+d_F(v)\le r_0+\Delta(F)$;
\item $v\in S\setminus B_1$: $d_G(v)\le r_0-|A_1|\le r_0-1$;
\item $u\in K$: $d_G(u)\ge r_0-1+|B_1|$, with equality on $A_1$.
\end{itemize}
Thus $d_G(K)\ge d_G(B_1)\ge d_G(S\setminus B_1)$, confirming $K$ as the
$r_0$ vertices of largest degree.

Now $d_{k+1}^*\le r_0$ means every vertex in $S$ has degree $\le k$.
The maximum degree in $S$ is $r_0+\Delta(F)$.
Hence $k\ge r_0+\Delta(F)$.
$r_0\le d_k^*$ means at least $r_0$ vertices have degree $\ge k$.
The minimum degree in $K$ is $r_0+|B_1|-1$. Hence $k\le r_0+|B_1|-1$.
Combining with the spectral constraint yields~\eqref{eq:typeII-range}.

\subsection{Proof of sufficiency}

\medskip\noindent\textit{Type~I sufficiency.}
Let $G$ be a Type~I graph with threshold graph $H$, and $k$ in~\eqref{eq:typeI-range}.
Set $r_0=\sum_i|U_i|$ and $Q=P_{(r_0,\infty)}(L_G)$.
From the spectral analysis above, $Q=P_{E_{>r_0}(L_H)}$.
We verify the five conditions of Proposition~\ref{prop:five-conditions}.

\emph{(i)}: $\dim E_{>r_0}(L_G)=|U_2|+\cdots+|U_m|$, and $r_0$ has
multiplicity $c(\overline F)-1$.
The range~\eqref{eq:typeI-range} ensures $\lambda_k\ge r_0\ge\lambda_{k+1}$.

\emph{(ii)}: The degree analysis in the necessity proof applies verbatim.

\emph{(iii)}: $G[S]=\varnothing$, so the condition is vacuous.

\emph{(iv)}: Non-edges of $G[K]$ lie in $U_1$.
For $i,j\in U_1$, $ij\notin E(F)$, we have $e_i-e_j\in Z_1$.
Since $Z_1\perp E_{>r_0}(L_H)$, $Q(e_i-e_j)=P_{E_{>r_0}(L_H)}(e_i-e_j)=0$.

\emph{(v)}: $q=\operatorname{rank}Q=|U_2|+\cdots+|U_m|<r_0$, so Case~1 of
Theorem~\ref{thm:split-trace} applies directly, giving
$\operatorname{tr}(Q(L_H-r_0I_n))=e_H(K,S)$.

\medskip\noindent\textit{Type~II sufficiency.}
Let $G$ be a Type~II graph, $r_0=\sum_i|A_i|$, with $k$ in~\eqref{eq:typeII-range}.
Set $Q=P_{(r_0,\infty)}(L_G)$. In $G$, the edges inside $B_1$ form $F$, and all other edges agree with $H$.
Thus $L_G=L_H+L_F$, with $L_F$ extended by zero outside $B_1$.
Every eigenvector of $L_H$ outside $\widetilde Z_1$ is either constant or zero
on $B_1$, hence is annihilated by $L_F$; its eigenvalue is unchanged.
On $\widetilde Z_1$, we have $L_H|_{\widetilde Z_1}=r_0I$ (since $B_1$ is
adjacent to all of $K$) and
$L_G|_{\widetilde Z_1}=r_0I+L_F|_{\widetilde Z_1}$.
The image of $Q$ consists of all eigenvectors of $L_G$ with eigenvalue $>r_0$:
the unchanged $E_{>r_0}(L_H)$ plus the eigenvectors in $\widetilde Z_1$
corresponding to positive eigenvalues of $L_F|_{\widetilde Z_1}$.
The latter are precisely $(\ker(L_F|_{\widetilde Z_1}))^\perp\cap\widetilde Z_1$.
Letting $\widetilde W=\ker(L_F|_{\widetilde Z_1})\subset \widetilde Z_1$, we have
$\dim\widetilde W=c(F)-1$, and
\[
Q=P_{E_{>r_0}(L_H)}+P_{\widetilde Z_1\cap\widetilde W^\perp},\qquad
\operatorname{rank}Q=r_0+|B_1|-c(F)\ge r_0.
\]
Equivalently,
$Q=I_n-P_{\mathbf 1}-P_{E_{<r_0}(L_H)}-P_{\widetilde W}$, matching Case~2 of
Theorem~\ref{thm:split-trace}.

\emph{(i)}:
From the spectral analysis above,
$\dim E_{>r_0}(L_G)=r_0+|B_1|-c(F)$.
The next $c(F)-1$ eigenvalues equal $r_0$ (coming from $\widetilde W$, where
$L_F$ acts as zero on $\widetilde Z_1$).
The range~\eqref{eq:typeII-range} therefore gives
$\lambda_k(G)\ge r_0\ge\lambda_{k+1}(G)$.

\emph{(ii)}:
The degree analysis in the necessity proof applies verbatim.

\emph{(iii)}:
Edges of $G[S]$ lie exclusively in $B_1$.
Take $ij\in E(F)$.
Then $e_i-e_j\in \widetilde Z_1\perp E_{<r_0}(L_H)$.
Also $e_i-e_j\perp\mathbf 1$, and since $i,j$ lie in the same connected
component of $F$, vectors in $\widetilde W$ are constant on that component,
so $e_i-e_j\perp\widetilde W$.
From $Q=I-P_{\mathbf 1}-P_{E_{<r_0}(L_H)}-P_{\widetilde W}$, we obtain
$Q(e_i-e_j)=e_i-e_j$.

\emph{(iv)}:
$G[K]=H[K]$ is a complete graph, so the condition is vacuous.

\emph{(v)}:
The projection $Q$ is of the form required by Case~2 of
Theorem~\ref{thm:split-trace}, with $q=r_0+|B_1|-c(F)\ge r_0$.
Therefore $\operatorname{tr}(Q(L_H-r_0I_n))=e_H(K,S)$.

All five conditions of Proposition~\ref{prop:five-conditions} are satisfied,
hence equality holds for $(G,k)$.

\subsection*{Acknowledgements}
This work is partly supported by the National Natural Science Foundation of China (No.12371354, W2521102), the Montenegrin-Chinese Science and Technology Cooperation Project (No.4-3)  and  the Science and Technology Commission of Shanghai Municipality (No.25LN3200600).

\end{document}